\documentclass[12pt]{amsart}
\usepackage{newlfont,amsmath,amssymb,enumerate}

\usepackage{amscd}
\usepackage{braket,xcolor}

\newcommand{\Sym}{{\mathrm{Sym}}}
\newcommand{\symm}{{\mathrm{symm}}}

\newcommand{\bfzero}{{\mathbf{0}}}
\newcommand{\bfone}{{\mathbf{1}}}

\newcommand{\rd}{\mathrm{d}}

\newcommand{\rpp}{t}

\newcommand{\ind}{{\mathrm{ind}}}

\newcommand{\thetaone}{{\theta_p^{n,m}(1)}}

\newcommand{\bbC}{{\mathbb C}}

\newcommand{\bbR}{{\mathbb R}}

\newcommand{\bbZ}{{\mathbb Z}}
\newcommand{\oo}{\widetilde{\mathrm O}}

\newcommand{\rO}{{\mathrm{O}}}
\newcommand{\rU}{{\mathrm{U}}}

\newcommand{\rSp}{{\mathrm{Sp}}}

\newcommand{\SO}{{\mathrm{SO}}}

\newcommand{\frakg}{{\mathfrak{g}}}

\newcommand{\frakn}{{\mathfrak{n}}}

\newcommand{\fraku}{{\mathfrak u}}
\newcommand{\frakt}{{\mathfrak{t}}}
\newcommand{\frakk}{{\mathfrak k}}
\newcommand{\frakl}{{\mathfrak{l}}}

\newcommand{\frakso}{{\mathfrak{so}}}

\newcommand{\frakgl}{{\mathfrak{gl}}}
\newcommand{\frakp}{{\mathfrak p}}
\newcommand{\frakq}{{\mathfrak q}}

\newcommand{\spl}{{\mathrm{Sp}}}

\newcommand{\syp}{\widetilde{\mathrm{Sp}}}

\newcommand{\Hom}{{\mathrm{Hom}}}
\newtheorem{lemma}[subsection]{Lemma}

\newtheorem{prop}[subsection]{Proposition}
\newtheorem{thm}[subsection]{Theorem}

\newcommand{\calB}{{\mathcal{B}}}
\newcommand{\calC}{{\mathcal{C}}}
\newcommand{\calF}{{\mathcal{F}}}
\newcommand{\calH}{{\mathcal{H}}}

\newcommand{\calL}{{\mathcal{L}}}

\newcommand{\rH}{{\mathrm{H}}}

\newcommand{\fgta}{{\calF_{\frakg,K_{n,m}}^{\frakg,K_{n,r,m-r}}}}

\newcommand{\fgonetb}{{\calF_{\frakg_1,K_{n,r}}^{\frakg_1,K_{n,t,r-t}}}}
\newcommand{\fgtc}{\calF_{\frakg_1,K_{n,r}}^{\frakg_1,K^0_{n,r}}}

\begin{document}


\title[Transfers of $K$-types]{Transfers of $K$-types on local theta
  lifts of characters and unitary lowest weight modules}

\author{Hung Yean Loke}

\author{Jia-jun Ma}

\author{U-Liang Tang}
\address{Department of Mathematics,
National University of Singapore,
2 Science Drive 2, Singapore 117543}

\email{matlhy@nus.edu.sg, jiajunma@nus.edu.sg, g0600577@nus.edu.sg}

\begin{abstract}
  In this paper we study representations of the indefinite orthogonal
  group $\rO(n,m)$ which are local theta lifts of one dimensional
  characters or unitary lowest weight modules of the double covers of
  the symplectic groups. We apply the transfer of $K$-types on these
  representations of $\rO(n,m)$, and we study their effects on the
  dual pair correspondences. These results provide examples that the
  theta lifting is compatible with the transfer of $K$-types.  Finally
  we will use these results to study subquotients of some
  cohomologically induced modules.
\end{abstract}


\maketitle

\section{Introduction} \label{S1} In this paper we study
representations of the indefinite orthogonal group $\rO(n,m)$ which
are local theta lifts of one dimensional characters or
unitary lowest weight modules of the double covers of the symplectic
groups. We apply the transfer of $K$-types on these
representations of $\rO(n,m)$, and we study their effects on the dual
pair correspondences. Finally we will use these results to study
subquotients of some cohomologically induced modules. Our methods
could also be applied to $\rU(n,m)$ and~$\rSp(n,m)$.

\smallskip

We introduce the Harish-Chandra module $\thetaone$ of $\rO(n,m)$.
Let $\syp(p(n+m),\bbR)$ be the metaplectic double cover of
$\spl(p(n+m),\bbR)$. It contains a dual pair $\left(\syp(p,\bbR),
\oo(n,m)\right)$. Here $\syp(p,\bbR)$ and $\oo(n,m)$ are (possibly split)
double covers of $\spl(p,\bbR)$ and $\rO(n,m)$ respectively.  A
Harish-Chandra module of $\syp(p,\bbR)$ (resp.  $\oo(n,m)$) is called
genuine if it does not factor through the linear group
$\spl(p,\bbR)$ (resp. $\rO(n,m)$).

We will always assume that $\syp(p,\bbR)$ splits over $\spl(p,\bbR)$,
i.e. $\syp(p,\bbR) \simeq \spl(p,\bbR) \times (\bbZ/ 2 \bbZ)$.  This
happens if and only if $m + n$ is even. Let $\varsigma'$ denote the
genuine one dimensional character of $\syp(p,\bbR)$ which is trivial
on $\spl(p,\bbR)$ and nontrivial on $\bbZ/ 2 \bbZ$. We fix an
oscillator representation of $\syp(p(n+m),\bbR)$ and we let $U_0$
denote the local theta lift of $\varsigma'$ to $\oo(n,m)$ with respect
to this oscillator representation \cite{H2}.  The module $U_0$ is an
irreducible Harish-Chandra module of $\oo(n,m)$. The group $\oo(n,m)$
splits over $\rO(n,m)$ if and only if $p$ is even. Since $\oo(n,m)$ has 8 connected components and it has four genuine
one dimensional characters.  We will explain a choice of a genuine
character $\varsigma$ in the paragraph before \eqref{eq6}.  We set
$\thetaone = \varsigma U_0$ which is an irreducible
Harish-Chandra module of $\rO(n,m)$.  We will call $\thetaone$
the theta lift of the trivial character of $\spl(p,\bbR)$.

We say that the dual pair $\left(\syp(p,\bbR), \oo(n,m)\right)$ is in the
{\it stable range} if $2p \leq \min(n,m)$ and $2p < \max(n,m)$. Our
definition excludes the case $m = n = 2p$. For simplicity, we will say
that $(p,n,m)$ is in stable range. By \cite{Li}, the Harish-Chandra
module $\thetaone$ is nonzero and unitarizable in the stable
range.

Let $K_{n,m} = \rO(n) \times \rO(m)$ denote a maximal compact subgroup
of $\rO(n,m)$.  The $K_{n,m}$-types of $\thetaone$ are described
in \eqref{eq6}.  It is $K_{n,m}$-multiplicity free and it is both
$\rO(n) \times 1$-admissible and $1 \times \rO(m)$-admissible. Let
$\frakg = \frakso(n+m,\bbC)$, let $K_{n,r,m-r} = \rO(n) \times \rO(r)
\times \rO(m-r)$ be the subgroup of $K_{n,m}$ and let $\fgta
\thetaone$
denote the restriction of $\thetaone$ as a
$(\frakg, K_{n,r,m-r})$-module. We will apply
the transfer of $K$-types due to Enright and Wallach \cite{Wa2}. More
precisely we apply the middle degree  $pr$-th derived functor $\Gamma^{pr} =
(\Gamma_{\frakg,K_{n,r,m-r}}^{\frakg,K_{n+r,m-r}})^{pr}$ of the
Zuckerman functor to $\fgta \thetaone$.  
Since
$K_{n+r,m-r}$ is not connected, special care is necessary. See
Section~\ref{S22} for a discussion. Let $\rho_N$ denote the half sum
of positive roots of $\frakso(N)$ and let $\bfone_p = (1, \ldots, 1)
\in \bbZ^p$. We can now state the first main theorem of this paper.

\begin{thm} \label{T11} 
Suppose $n + m$ is even and $(p,n,m)$ is in the stable range. Let $0 <
r < m$.
\begin{enumerate}[(i)]
\item
If $2p > m-r$, then $\Gamma^{pr}(\fgta
 \thetaone) =  0$.

\item
If $2p \leq m-r$, then
\[
\Gamma^{pr}(\fgta( \thetaone)) =
\theta_p^{n+r,m-r}(1)
\]
as $(\frakso(n+m,\bbC), K_{n+r,m-r})$-modules.

\item We set $N = n + m$. Then every representation in the collection
\[
\{ \theta_p^{a,b}(1) : a + b = N, a \geq 2p, b \geq 2p \}
\]
has the same annihilator ideal in the universal enveloping algebra of
$\frakso(n+m,\bbC)$. The annihilator ideal is the maximal primitive
ideal with infinitesimal character $\rho_{N} - p
\bfone_{N/2}$. See Section \ref{S21} for the notation on
weights and infinitesimal characters.

\item Let $a = a_1 + a_2$ and $b = b_1 + b_2$. Suppose $\theta_p^{a_1+b_1,a_2+b_2}(1)
  \rightarrow \pi_a \boxtimes \pi_b$ is a nontrivial quotient where
  $\pi_a$ and $\pi_b$ are irreducible modules of $(\frakso(a,\bbC), K_{a_1,a_2})$ and $(\frakso(b,\bbC),
  K_{b_1,b_2})$ respectively. Then the infinitesimal characters of
  $\pi_a$ and $\pi_b$ must respect the following correspondence:
\begin{eqnarray}
(\lambda_1, \ldots, \lambda_p, \rho_{a-2p}) \longleftrightarrow
(\lambda_1, \ldots, \lambda_p, \rho_{b-2p}) & & \mbox{ if } b \geq a
\geq 2p, \label{eq1} \\
 (\lambda_1, \ldots, \lambda_s) \longleftrightarrow (\lambda_1,
 \ldots, \lambda_s, \rho_{b-a} - \frac{2p-a}{2} \bfone_{\frac{b-a}{2}})
\label{eq2} & & \mbox{ if } a < 2p
\end{eqnarray}
where $s = [\frac{a}{2}]$.
\end{enumerate}
\end{thm}

The module $\thetaone$ has been investigated in a number of
papers \cite{HZ}, \cite{KO}, \cite{LZ}, \cite{Z}.

\medskip

Irreducible representations of $\rO(r')$ are parameterized by certain
arrays of nonnegative integers $\mu = (\mu_1, \ldots, \mu_{r'}) \in
\bbZ^{r'}$. See Section \ref{S21}. We will use $\mu_{\rO(r')}$ or
simply $\mu$ to denote the corresponding irreducible finite
dimensional representation of $\rO(r')$.  Then the local theta lift
$L(\mu)$ of $\mu$ to a (possibly split) double cover
$\widetilde{\rSp}(p,\bbR)$ is a unitarizable lowest weight module.  A
result of \cite{EHW} states that conversely, a unitarizable lowest
weight module of the connected component of a double cover of
$\rSp(p,\bbR)$ is the restriction of a unique $L(\mu)$.  Let $U_1$
be the local theta lift of $L(\mu)$ to $\oo(n,r)$. If it is nonzero,
we will choose a genuine character $\varsigma''$ of $\oo(n,r)$
so that $\theta_p^{n,r}(L(\mu)) := \varsigma'' U_1$ is a
Harish-Chandra module of $\rO(n,r)$ and it has a $K_{n,r}$-type
decomposition as in \eqref{eq8}. We remark that
$\theta_p^{n,r}(L(\mu))$ is $\rO(n) \times 1$-admissible but it is
almost never $K$-multiplicity free. Let $\frakg_1 =
\frakso(n+r,\bbC)$. We will state our second main theorem.

\begin{thm} \label{T12} 
  Suppose $(p,n, r + r')$ is in the stable range.  Let $\mu$ and
  $\theta_p^{n,r}(L(\mu))$ as above. Let $0 < \rpp< r$. Let
  $\Gamma_1^{pt} =
  (\Gamma_{\frakg_1,K_{n,t,r-t}}^{\frakg_1,K_{n+t,r-t}})^{p\rpp}$
  denote the $pt$-th derived Zuckerman functor.
\begin{enumerate}[(i)]
\item If $2p > r+r'-\rpp$, then $\Gamma_1^{p\rpp}(\fgonetb
  \theta_p^{n,r}(L(\mu))) = 0$.
 
\item If $2p \leq r+r'-\rpp$, then
\[
\Gamma_1^{p\rpp}(\fgonetb \theta_p^{n,r}(L(\mu))) =
  \theta_p^{n+\rpp,r-\rpp}(L(\mu))
\]
as $(\frakg_1, K_{n+r,r-\rpp})$-modules.
\end{enumerate}
\end{thm}

In (ii), it is possible that $\theta_p^{n+\rpp,r-\rpp}(L(\mu)) = 0$
and when this happens, the above theorem says that the left hand side
of (ii) is also zero. If we set $r' = 0$ in the above theorem, then we
recover parts (i) and (ii) of Theorem~\ref{T11}. However the proof of
Theorem \ref{T12} requires Theorem \ref{T11}.

We will prove Theorems \ref{T11} and \ref{T12} in Section \ref{S2}.

In Theorems \ref{T11} and \ref{T12}, we have assumed that $(p,n,m)$ and
$(p,n,r+r')$ lie in the stable range. In Section~\ref{S34}, we will
address the situation when we are outside of this range.

We will briefly explain our motivation. It is a well known fact that
the correspondences of the infinitesimal characters induced by the
local theta lifts are independent of the real forms of the dual pair
\cite{H1} \cite{Pz}. We expect that the local theta lifts of a
representation to different real forms share many more interesting
properties which are related by the transfer of $K$-types.  For
example, Conjecture 5.1 in \cite{WZ} predicts an explicit relationship
between the transfers of $K$-types and the local theta lifts of
characters of $\rO(n,m)$ to $\widetilde{\rSp}(p,\bbR)$.  In a forth
coming paper of the second author, J.-j. Ma will show that theta
lifts of one dimensional representations and the transfers of
$K$-types are compatible operations and extend Theorems~\ref{T11} and
\ref{T12} to dual pairs $(\rU(n,m), \rU(p,q))$ and
$(\rSp(n,m),\rO^*(2p))$. Also see \cite{Ta}.

\subsection{Cohomological inductions.} \label{S13}
We will apply the above two theorems to some cohomological induced
modules.  In order to state our next main result, we recall some basic
definitions and notation from \cite{KV} and \cite{Wa1}.

We suppose $2p \leq n \leq m$, $2p < m$ and $m = r + r'$.  Let
$\frakg_1 = \frakso(n,r) \otimes \bbC = \frakso(n+r,\bbC)$.  Let
$\frakt_0$ be a compact Cartan subalgebra of $\frakso(n) \oplus
\frakso(r)$ in $\frakso(n,r)$. Let $\lambda_0' = (p,p-1,\ldots,1, 0, 0,
\ldots, 0) \in \sqrt{-1}\frakt_0^{\ast}$. Let $\frakq' =\frakl' +
\frakn'$ be the maximal $\theta$-stable parabolic subalgebra in
$\frakg_1$ defined by $\lambda_0'$. The real form of the Levi
subalgebra is $\frakl_0' = \fraku(1)^p \oplus \frakso(n-2p,m)$. Let
$L' = \rU(1)^p \times \rO(n-2p,r)$ be the subgroup in $\rO(n,r)$ whose
Lie algebra is $\frakl_0'$.  Let $\bbC_\lambda$ be a character of
$\rU(1)^p$. We extend it to a character of $L'$ such that
$\rO(n-2p,r)$ acts on it by $\det_{\rO(n-2p,r)}^p$. Let $\frakg_1 =
\frakk_{n,r} \oplus \frakp_{n,r}$ denote the complexification of the
Cartan decomposition of $\frakso(n,r)$. Let $M' = L' \cap K_{n,r} =
\rU(1)^p \times K_{n-2p,r}$ and
\[
s_0' = \dim (\frakk_{n,r} \cap
\frakn') = p(n-p-1).
\]
We refer to (5.3a) on page 328 in \cite{KV} and define a $(\frakg_1,
K_{n,m})$-module
\[
A(\lambda) = \calL_{s_0'}(\bbC_{\lambda}) =
(P_{\overline{\frakq'},M'}^{\frakg_1,K_{n,r}})_{s_0'}
(\calF_{\frakl',M'}^{\overline{\frakq'},M'}
(\bbC_\lambda^\sharp)),
\]
where $P_{\overline{\frakq'},M'}^{\frakg_1,K_{n,r}}$ is the induction
functor in Section II.1 in \cite{KV} and
$(P_{\overline{\frakq'},M'}^{\frakg_1,K_{n,r}})_{s_0'}$ is its
$s_0'$-th derived functor. Also see Section \ref{S4}. The module
$A(\lambda)$ has infinitesimal character $\lambda + \rho_{n+r}$.

Given a $\lambda$ as in Theorem \ref{T14} below, we will show in Lemma
\ref{L41} that the bottom layer $K_{n,r}$-type is the minimal
$K_{n,r}$-type of $A(\lambda)$. Let $\overline{A}(\lambda)$ denote the
irreducible subquotient of $A(\lambda)$ generated by the minimal
$K_{n,r}$-type. We can now state our main result on cohomological inductions.

\begin{thm}\label{T14}
  Suppose $2p \leq n \leq m$, $2p < m$.
\begin{enumerate}[(i)]
\item The irreducible $(\frakso(n+m,\bbC),K_{n,m})$-modules
  $\thetaone$ and $\overline{A}(\lambda)$ are
  isomorphic where $\lambda = -\frac{m+n}{2} \bfone_p$.
  
\item Let $0 \leq r < m$ and let $\mu = (\mu_1, \ldots, \mu_{r'})$
  such that $\mu_i = 0$ if $i > p$.  Then the irreducible
  $(\frakso(n+r,\bbC),K_{n,r})$-modules $\theta_p^{n,r}(L(\mu))$ and
  $\overline{A}(\lambda)$ are isomorphic where $\lambda = (\mu_1,
  \ldots, \mu_p) + \frac{m-n-2r}{2} \bfone_p$.
\end{enumerate}
\end{thm}

In (i), we set $r = m$. In (ii), both $L(\mu)$ and $\theta_p^{n,r}(L(\mu))$ 
are zero if $\mu_{p+1} > 0$. 

We refer to page 330 in \cite{KV} for the definition of the
$(\frakg_1,K_{n,m}^0)$-module $A_{\frakq'}(\lambda) =
\calL_{s_0'}(\bbC_\lambda)$.  Let $\fgtc A(\lambda)$ denote the restriction
of $A(\lambda)$ as a $(\frakg_1,K_{n,m}^0)$-module. We will show in
Proposition \ref{P42} that $\fgtc A(\lambda)$ contains the
$(\frakg_1,K_{n,m}^0)$-module $A_{\frakq'}(\lambda)$ as a
submodule. If $n > 2p$, then $\fgtc
  A(\lambda)  = A_{\frakq'}(\lambda)$.
We remark that $A_{\frakq'}(\lambda)$ in the above theorem is not
always in the good or weakly good range (see Definition 0.49 in
\cite{KV}).  It is interesting to find unitarizable subquotients
generated by the images of the bottom layer maps.

We will explain the motivation of the above theorem. First we consider
the special case when $n = 2p < r$, $\mu = 0$ and $\lambda =
\frac{m-n-2r}{2} \bfone_p$. In \cite{Kn}, Knapp constructed a
unitarizable quotient $A'$ of $A_{\frakq'}(\lambda)$ by extending the
method of Gross and Wallach \cite{GW}.  The quotient module $A'$
contains the image of the bottom layer map. He asked if $A'$ is
irreducible and if $A'$ is related to local theta lifts.  The first
question was answered by Trapa where he showed that $A'$ is
irreducible \cite{T}. He also computed its associated cycle. In
\cite{PT}, Trapa and Paul show that $A'$ is a submodule of
$\theta_p^{2p,r}(L(0))$. Theorem \ref{T14} could be considered a
generalization of these results.

\subsection*{Acknowledgment}
We would like to thank the referee for the suggestion to use Zuckerman
functors for disconnected groups, which simplifies many proofs. We
thank Peter Trapa for directing our attention to the treatment of
disconnected groups in \cite{KV}.  We also thank Chen-bo Zhu for
helpful discussions and comments on the manuscript. The first author
is supported by an NUS grant R-146-000-131-112.

\section{The Zuckerman functors} 

The objectives of this section are to set up some notation and define
the Zuckerman functors in Theorems~\ref{T11} and \ref{T12}.

\subsection{Weights} \label{S21}

Let ${\mathbf{1}}_n := (1,1,\ldots,1)$ and ${\mathbf{0}}_n :=
(0,0,\ldots,0)$ in $\bbR^n$.  If $\lambda = (\lambda_1, \ldots,
\lambda_n) \in \bbR^n$ and $\xi = (\xi_1, \ldots, \xi_n) \in \bbR^m$,
then we denote $(\lambda_1, \ldots, \lambda_n, \xi_1, \ldots, \xi_m)
\in \bbR^{n+m}$ by $(\lambda, \xi)$.

Let $r = [\frac{n}{2}]$ and let $\rho_n \in \bbR^r$ denote the half
sum of positive roots of $\frakso(n,\bbC)$.  Let $\mu \in
  \bbR^r$ be a highest weight of $\frakso(n,\bbC)$. We will use
  $\mu_{\frakso(n)}$ or simply $\mu$ to denote the corresponding
  irreducible finite dimensional representation of $\frakso(n)$.

We will denote the trivial and the determinant representation of
$\rO(n)$ by $\bbC_{\rO(n)}$ and $\rd_n$ respectively.  Irreducible
representations of $\rO(n)$ are parameterized by arrays of the form
\begin{equation} \label{eq3}
\mu = (a_1, \ldots, a_k, \bfzero_{n-k})  \mbox{ or }
(a_1, \ldots, a_k, \bfone_{n-2k},\bfzero_{k})
\end{equation}
in $\bbZ^n$ where $a_i$ are positive integers, $a_i \geq a_{i+1}$ and
$k \leq [\frac{n}{2}]$. See \cite{GoW}. We will call these arrays {\it
  weights} of $\rO(n)$.  Let $\Lambda(\rO(n))$ denote the set of such
weights. Given such a weight~$\mu$, we will use $\mu_{\rO(n)}$ or
simply $\mu$ denote the corresponding irreducible finite dimensional
representation of $\rO(n)$.  The two representations corresponding to
the two highest weights in \eqref{eq3} differ by the character
$\rd_n$. In general, we would ignore the string of zeros at the end of
a weight $(a_1, \ldots, a_k, \bfzero_{n-k})$ and write it as $(a_1,
\ldots, a_k)$ instead.

Finally we recall a branching rule: The dimension of
$\Hom_{\rO(n)}(\mu_{\rO(n)}, \mu_{\rO(n+1)}')$ is at most one and it
is one if and only if $\mu_i' \geq \mu_i \geq \mu_{i+1}'$ for all $i =
1, \ldots, n$.

\medskip

Throughout this paper, we denote $K_n = \rO(n)$, $K_{n,m} = \rO(n) \times \rO(m)$,
$K_{n,r,r'} = \rO(n) \times \rO(r) \times \rO(r')$ and $K_{n,m}^0 =
\SO(n) \times \SO(m)$. Let $\frakg_1 = \frakso(n+r,\bbC)$. 
  Given an irreducible $(\frakg_1,K_{n,r})$-module, we will follow
  Harish-Chandra parametrization and use a weight $\mu$ of $\frakg_1$
  to denote its infinitesimal character. We note that the
  infinitesimal character is a character of
  $U(\frakg_1)^{\rO(n+r,\bbC)}$.  Two weights $\mu$ and $\mu'$ give
  the same infinitesimal character if and only if $\mu = w \mu'$ for
  some $w$ in the Weyl group of $\rO(n+r,\bbC)$.

  \subsection{Zuckerman functors} \label{S22} Let $W$ be a $(\frakg_1,
  K_{n,r})$-module. We will follow Section II.8.5 in \cite{BW} and
  \cite{Wa1} where it is established that $i$-th Zuckerman functor
\begin{equation}
\Gamma^i(W) = (\Gamma_{\frakg_1,K_{n,r}}^{\frakg_1,K_{n+r}})^i(W) =
\bigoplus_F \rH^i(\frakg_1,K_{n,r}; W \otimes F^*) \otimes F. \label{eq4}
\end{equation}
Here the sum is taken over all irreducible finite dimensional
representations $F$ of $\rO(n+r)$. Hence $\Gamma^i(W)$ is nonzero if
and only if the Lie algebra cohomology
\[
\rH^i(\frakg_1,K_{n,r}; W \otimes F^*)
\]
is nonzero for some $F$. Suppose the above Lie algebra cohomology is
nonzero, then by Wigner's lemma, $W$ and $F$ have the same
infinitesimal character. In particular $W$ has regular infinitesimal
character.  Since $K_{n,r}$ contains the center of
$\rO(n,r)$, $W$ and $F$ have the same central character, again by
  Wigner's lemma.  Let
$\frakg_1 = \frakk_{n,r} \oplus \frakp_{n,r}$ be the complexified
Cartan decomposition. By Proposition 9.4.3 in \cite{Wa1}, if $W$ is
unitarizable, then
\begin{equation}
\rH^i(\frakg_1,K_{n,r}; W \otimes F^*) =
\Hom_{K_{n,r}}(\wedge^i \frakp_{n,r}, W \otimes F^*)
 = \Hom_{K_{n,r}}(F \otimes \wedge^i \frakp_{n,r}, W).
\label{eq5}
\end{equation}
We remark that \cite{Wa1} requires that the maximal compact subgroup
is connected. However the proof there works for $K_{n,r}$ too
without any modification.

Let $\frakg = \frakso(n+r+r',\bbC)$. Suppose $W$ is a $(\frakg,
K_{n,r,r'})$-module. Let $\calF W :=
\calF_{\frakg,K_{n,r,r'}}^{\frakg_1,K_{n,r}} W$ denote the restriction
of $W$ as a $(\frakg_1,K_{n,r})$-module. By the naturality property
of the Zuckerman functor in Chapter 6 of \cite{Wa1}, the $(\frakg,
K_{n,r,r'})$-module structure on $W$ gives a natural $(\frakg,
K_{n+r,r'})$-module structure on $\Gamma^i(\calF W)$. This
$(\frakg, K_{n+r,r'})$-module $\Gamma^i(\calF W)$ is
isomorphic to $(\Gamma_{\frakg,K_{n,r,r'}}^{\frakg,K_{n+r,r'}})^i W$
by Section II.8.5(1) in \cite{BW}. In other words,
$(\Gamma_{\frakg,K_{n,r,r'}}^{\frakg,K_{n+r,r'}})^i(W)$ is computed by
$\Gamma^i(\calF W)$ in the category of $(\frakg_1,K_{n,r})$-modules.

\section{Proofs of Theorems \ref{T11} and \ref{T12}} \label{S2}

First we recall some facts about the local theta lift $U_0$ of the
character $\varsigma'$ of $\syp(p,\bbR)$ in \cite{Lo}. We will assume
that $n + m$ is even and $(p,n,m)$ is in the stable range.  There are
four choices of the genuine character $\varsigma$ of
$\oo(n,m)$ as mentioned in the introduction.  We will fix a
choice of $\varsigma$ so that $\thetaone = \varsigma
U_0$ has $K_{n,m}$-types
\begin{equation}
  \thetaone = \bigoplus_{l = (l_1, \ldots, l_p)} \left(\rd_n^p \, (l +
  \frac{m-n}{2} \bfone_p, \bfzero_{n-p})_{\rO(n)} \right) \boxtimes (l,
  \bfzero_{m-p})_{\rO(m)} \label{eq6}
\end{equation}
where $\rd_n = \det_{\rO(n)}$ and the sum is taken over $l =
(l_1,\ldots, l_p)$ such that $l_i$ are non-negative integers and $l_1
\geq l_2 \geq \ldots \geq l_p \geq \min(\frac{n-m}{2},0)$.  The module
$\thetaone$ is unitarizable. The minimal $K_{n,m}$-type
$\tau_{\min}$ is $\rd_n^p (\frac{m-n}{2} \bfone_p,
\bfzero_{n-p})_{\rO(n)} \boxtimes \bbC_{\rO(m)}$ if $m \geq n$ and
$\rd_n^p \boxtimes (\frac{n-m}{2} \bfone_p, \bfzero_{m-p})_{\rO(m)}$
if $m < n$.

Let $\frakg = \frakso(n+m,\bbC)$ and $K_{n,m}^0 = \SO(n) \times \SO(m)$. If $\min(n,m) = 2p$ then
$\thetaone$ splits into a direct sum of two irreducible
$(\frakg,K_{n,m}^0)$-modules. If $n,m > 2p$, then it is an irreducible
$(\frakg,K_{n,m}^0)$-module.

Let $m = r + r'$ and let $\frakg_1 = \frakso(n+r,\bbC)$. Since
$\thetaone$ is unitarizable and $\rO(n) \times 1$-admissible, the
restriction of $\thetaone$ as a 
$(\frakg_1, K_{n,r}) \times \rO(r')$-module decomposes discretely as a direct
sum 
\begin{equation}
\thetaone = \bigoplus_{\mu \in \Lambda(\rO(r'))} \Omega(\mu)
\boxtimes \mu_{\rO(r')}.
\label{eq7}
\end{equation}
It follows that $\Omega(\mu)$ is unitarizable and it has
$K_{n,r}$-types
\begin{equation} \label{eq8}
\Omega(\mu) = \bigoplus_{l = (l_1, \ldots, l_p)}
 \left( \rd_n^p \, (l + \frac{m-n}{2} \bfone_p,
     \bfzero_{n-p})_{\rO(n)} \boxtimes \bigoplus_{\kappa}
  m(l,\mu,\kappa) \kappa_{\rO(r)} \right)
\end{equation}
where the sum is as in \eqref{eq6} and $m(l,\mu,\kappa)$ is the
multiplicity of $\kappa_{\rO(r)} \boxtimes \mu_{\rO(r')}$
in $(l, \bfzero_{m-p})_{\rO(m)}$.

Suppose $k \leq p$ and $k\leq r'$.  We lift $\mu_{\rO(r')} = (\mu_1,
\ldots, \mu_k, \bfzero_{r'-k})_{\rO(r')}$ to a unitarizable lowest
weight module $L(\mu)$ of a double cover of $\rSp(p,\bbR)$.  The
lowest $\widetilde{\rU}(p)$-type has highest weight (see~\cite{KaV})
\[
(\mu_1 \ldots, \mu_k, \bfzero_{p-k}) + \frac{r'}{2} \bfone_p.
\]
Let $U_1$ be the local theta lift of $L(\mu)$ to $\oo(n,r)$. In
\cite{Lo} we proved that $U_1$ is an irreducible and unitarizable
Harish-Chandra module of $\oo(n,r)$. There exists a genuine
character $\varsigma''$ of $\oo(n,r)$ such that
$\Omega(\mu) = \varsigma'' U_1$. Furthermore
$\Omega(\mu) \neq 0$ if and only if $U_1 \neq 0$.
From now on, we will denote $\Omega(\mu)$ by $\theta_p^{n,r}(L(\mu))$.

Let $\delta = \frac{n-m+2r}{2}$. Let $s = [r'/2]$ and let $\lambda =
(\lambda_1, \ldots, \lambda_s)$ denote the infinitesimal character of
$\mu_{\rO(r')}$.

\begin{lemma} \label{L31}
Suppose $(p,n, m)$ is in stable range and $m = r +
r'$. 
\begin{enumerate}[(i)]
\item If $r' \geq 2p$, $\theta_p^{n,r}(L(\mu))$ has infinitesimal
  character $(\mu - \delta \bfone_p, \bfzero) + \rho_{n+r}$. The
  infinitesimal character is regular if and only if $\mu_p \geq
  \delta$.

\item If $r' < 2p$, then $\theta_p^{n,r}(L(\mu))$ has infinitesimal
  character $(\lambda_1, \ldots, \lambda_s, \rho_{2 \delta} -
  \frac{2p-r'}{2} \bfone_{\delta})$. The infinitesimal character is
  singular.
\end{enumerate}
\end{lemma}

\begin{proof}
  This follows from the correspondences of infinitesimal characters
  $\frakso(r',\bbC) \leftrightarrow {\mathfrak{sp}}(p,\bbC)$ and
  ${\mathfrak{sp}}(p,\bbC) \leftrightarrow \frakso(n+r,\bbC)$
  established by the oscillator representations \cite{H1} \cite{Pz}.
  We will leave the details to the reader. 
\end{proof}

\medskip

We apply $(\Gamma_{\frakg,K_{n,r,r'}}^{\frakg,K_{n+r,r'}})^{pr}$ 
defined in Section \ref{S22} to \eqref{eq7} and we get
\begin{equation} \label{eq9}
(\Gamma_{\frakg,K_{n,r,r'}}^{\frakg,K_{n+r,r'}})^{pr}(\calF
 \thetaone) = \bigoplus_{\mu \in \Lambda(\rO(r'))}
(\Gamma_{\frakg_1,K_{n,r}}^{\frakg_1,K_{n+r}})^{pr}(\theta_p^{n,r}(L(\mu)))
 \boxtimes \mu_{\rO(r')}
\end{equation}
where $\calF = \calF_{\frakg,K_{n,m}}^{\frakg,K_{n,r,r'}}$ is the forgetful functor. 
From now on, we will write $\calF$ instead of $\calF_{\frakg,K_{n,m}}^{\frakg,K_{n,r,r'}}$ if $(\frakg,K_{n,m})$ 
and $(\frakg,K_{n,r,r'})$ could be inferred from the equation.
As explained in Section \ref{S22}, the equality in the above equation
holds by the naturality the Zuckerman functor in Chapter 6 in
\cite{Wa1}.

By Lemma \ref{L31}(ii), if $r' < 2p$ then $\theta_p^{n,r}(L(\mu))$ has
singular infinitesimal characters and
$(\Gamma_{\frakg_1,K_{n,r}}^{\frakg_1,K_{n+r}})^{pr}
(\theta_p^{n,r}(L(\mu))) = 0$.  This proves Theorem \ref{T11}(i).

\medskip

From now on, we will assume that $n, r' \geq 2p$ and $\mu_p \geq \delta$.
Let $\frakt$ be a Cartan subalgebra of the complexified Lie algebra
$\frakk_{n,r}$ of $K_{n,r}$. We set
\begin{equation}
\xi = \xi(\mu) := (\mu - \delta \bfone_p, \bfzero_{n+r-p})
\label{eq10}
\end{equation}
and $\bar{\xi} = \bar{\xi}(\mu) := (\mu - \delta \bfone_p, \bfzero) \in
\frakt^*$. Then both $\theta_p^{n,r}(L(\mu))$ and $\xi(\mu)_{\rO(n+r)}$
have regular infinitesimal character $\bar{\xi}(\mu) + \rho_{n+r}$.

We refer to the $K_{n,r}$-types of $\theta_p^{n,r}(L(\mu))$ in
\eqref{eq8}. By the branching rule described in Section~\ref{S21}, we
see that the minimal $K_{n,r}$-type of $\theta_p^{n,r}(L(\mu))$ is
\begin{equation}
\tau_{\min}(\mu) :=  \rd_n^p (\mu + \frac{m-n}{2} \bfone_p,
    \bfzero_{n-p})_{\rO(n)} \boxtimes \bbC_{\rO(r)}. \label{eq11}
\end{equation}

Proposition \ref{P32} below computes
$(\Gamma_{\frakg_1,K_{n,r}}^{\frakg_1,K_{n+r}})^{pr}
(\theta_p^{n,r}(L(\mu)))$. For connected groups, this is a special
case of Section 5 in \cite{VZ} or Theorem 9.5.3 in \cite{Wa1}. We will
give a proof for completeness.

\begin{prop} \label{P32} 
If $\theta_p^{n,r}(L(\mu))$ has regular infinitesimal character $\bar{\xi}(\mu) +
\rho_{n+r}$ , then
\[(\Gamma_{\frakg_1,K_{n,r}}^{\frakg_1,K_{n+r}})^{pr}
(\theta_p^{n,r}(L(\mu))) = \rd_{n+r}^p \xi(\mu)_{\rO(n+r)}.
\]
\end{prop}

\begin{proof}
  Let $\frakso(n+r,\bbC) = \frakk \oplus \frakp$ be the complexified
  Cartan decomposition.  We set $i = pr$ and $W =
  \theta_p^{n,r}(L(\mu))$ in \eqref{eq4} and \eqref{eq5}.  By Section
  \ref{S22}, \eqref{eq5} is nonzero only if $F$ is an irreducible
  finite dimensional representation of $\rO(n+r)$ with the same
  infinitesimal character as~$\theta_p^{n,r}(L(\mu))$. Hence we may
  assume $F$ is either $\xi(\mu)_{\rO(n+r)}$ or $\rd_{n+r}
  \xi(\mu)_{\rO(n+r)}$. Now it is suffices to show that
\begin{equation}
\Hom_{K_{n,r}}(F \otimes \wedge^{pr} \frakp,
\theta_p^{n,r}(L(\mu))) \label{eq12}
\end{equation}
is one dimensional if $F = \rd_{n+r}^p \xi(\mu)_{\rO(n+r)}$, and zero
if otherwise. 

We note that $F$ restricts irreducibly to $\SO(n+r)$ so it is a
quotient of the generalized Verma module $U(\frakg_1)
\otimes_{\frakq'} \bbC_{\overline{\xi}} = U(\overline{\frakn'})
\otimes_\bbC \bbC_{\overline{\xi}}$. Hence a $\frakt$-weight of $F$ is
of the form \linebreak $\mu - \delta \bfone_p - \sum \alpha$ where
$\sum \alpha$ is a sum some positive roots in $\frakn'$.

Let $\phi$ be a nonzero homomorphism of \eqref{eq12}. Let $\tau =
\rd_n^p(l + \frac{m-n}{2} \bfone_p)_{\rO(n)} \otimes \kappa_{\rO(r)}$
be a $K_{n,r}$-type of $\theta_p^{n,r}(L(\mu))$ in
\eqref{eq8}. Suppose $\tau$ is in the image of $\phi$. 
  Then the image contains the $\frakt$-weight $l + \frac{m-n}{2}
  \bfone_p + \bar{\kappa}$ where $\bar{\kappa}$ denotes a highest
  $\frakt$-weight of $\kappa_{\rO(r)}$. There is a $\frakt$-weight
$\mu - \delta \bfone_p - \sum \alpha$ of $F$ and a $\frakt$-weight $\sum
\beta + \sum \gamma - \sum \gamma'$ of $\wedge^{pr} \frakp$ such that
\begin{equation} \label{eq13}
\mu - \delta \bfone_p - \sum \alpha + \sum \beta + \sum \gamma - \sum
\gamma' = l + \frac{m-n}{2} \bfone_p + \bar{\kappa}.
\end{equation}
Here $\sum \beta$ is a sum of distinct roots in $\frakp \cap \frakl'$
and, $\sum \gamma$ and $\sum \gamma'$ are sums of distinct roots in
$\frakp \cap \frakn'$. Taking inner products with $\eta = (\bfone_p,
\bfzero)$ on both sides and rearrange the terms give
\[
\sum_{i=1}^p (l_i - \mu_i) + \langle \eta, \sum \alpha \rangle = -\left(pr-
\langle \eta, \sum \gamma \rangle \right) - \langle \eta, \sum \gamma' \rangle
 \leq 0.
\]
By the description of the $K_{n,r}$-types in \eqref{eq11}, the left hand side is non-negative.
Hence the left hand side is zero. This gives $l =
\mu$, $\sum \alpha = 0$, $\sum \gamma' = 0$, $\sum \gamma = r
\bfone_p$.  Since $l = \mu$, $\tau$ is the minimal $K_{n,r}$-type
$\tau_{\min}(\mu)$ and $\kappa = 0$. Putting these back into
\eqref{eq13} gives $\sum \beta = 0$.

Let $v_F$ be a highest weight vector of $F$ and let $v$ be a nonzero
vector in $\wedge^{pr} (\frakp \cap \frakn') = \wedge^{\mathrm{top}} (\frakp \cap \frakn')$. Then $v_F \otimes v$ is
a highest weight vector with respect to $\frakso(n) \oplus \frakso(r)$
in $F \otimes \wedge^{pr} \frakp$. Let $F'$ denote the irreducible
$K_{n,r}$-submodule generated by $v_F \otimes v$.  Now \eqref{eq12} becomes
\[
\Hom_{K_{n,r}}(F \otimes \wedge^{pr} \frakp, \theta_p^{n,r}(L(\mu))) =
\Hom_{K_{n,r}}(F', \tau_{\min}(\mu)). 
\]
Indeed the $\frakt$-weights argument above showed that the right hand
side contains the left hand side. The left hand side clearly contains
the right hand side. Finally the right hand side has dimension 0 or 1. 
It is 1 if and only if $F = \rd_{n+r}^p
\xi(\mu)_{\rO(n+r)}$. This proves the lemma.
\end{proof}

Before completing the proof of Theorem \ref{T11}, we need a
characterization of $\thetaone$.

\begin{lemma} \label{L33} 
Suppose $(p,n, m)$ is in stable range. Let $W_0$ be a Harish-Chandra
module of $\rO(n,m)$ which has the same infinitesimal character and
$K_{n,m}$-types as $\thetaone$. If $n,m > 2p$, then we further
assume that $W_0$ respects the correspondence of infinitesimal
characters in Lemma~\ref{L31}(i) for the pairs $\frakso(n+1) \times
\frakso(m-1)$ and $\frakso(n-1) \times \frakso(m+1)$. Then $W_0$ is
isomorphic to $\thetaone$.
\end{lemma}

\begin{proof}
  If $m = n$, then $\theta_p^{n,n}(1)$ is a representation with scalar
  $K$-type by \eqref{eq6}. Hence it is uniquely determined by its
  infinitesimal character and $K$-types. For example see
  \cite{Zhuscaler}.  This proves that $\theta_p^{n,n}(1) = W_0$.

We now suppose that $m > n \geq 2p$.  We consider the restriction of
$W_0$ as an \linebreak $(\frakso(n+1,\bbC), K_{n,1}) \times
\rO(m-1)$-module:
\[
W_0 = \bigoplus_{\mu \in \Lambda(\rO(m-1))} W_0(\mu) \boxtimes
\mu_{\rO(m-1)}.
\]
We claim $W_0(\mu)$ and $\theta_p^{n,1}(L(\mu))$ are isomorphic
$(\frakso(n+1,\bbC), K_{n,1})$-modules. Indeed by our hypothesis,
$W_0(\mu)$ has the same $K_{n,1}$-types as
$\theta_p^{n,1}(L(\mu))$. If $n = 2p$, then $\theta_p^{n,1}(L(\mu))$
is an irreducible discrete series representation. It is a well known
result of Schmid that an irreducible discrete series representation is
uniquely determined by its $K_{n,1}$-types \cite{Sc}. Hence $W_0(\mu)
\simeq \theta_p^{n,1}(L(\mu))$. If $n > 2p$, then
$\theta_p^{n,1}(L(\mu))$ is an irreducible
$(\frakso(n+1,\bbC),\SO(n))$-module.  If we set $r = 1$, then $2p < n
+ 1 \leq m -1$ and Lemma \ref{L31}(i) applies. Hence $W_0(\mu)$ has
the same infinitesimal character as $\theta_p^{n,1}(L(\mu))$. It is
well known that irreducible $(\frakso(n+1,\bbC),\SO(n))$-modules are
characterized by its infinitesimal characters and $\SO(n)$-types. This
implies that $\theta_p^{n,1}(L(\mu))$ and $W_0$ are isomorphic
$(\frakso(n+1,\bbC),\SO(n))$-modules. They extend to the same
$K_{n,1}$-module because they have the same $K_{n,1}$-types. This
proves our claim.

Let $\theta = \thetaone$. Let $\tau^l = \rd_n^p (l +
\frac{m-n}{2} \bfone_p, \bfzero_{n-p})_{\rO(n)} \boxtimes (l,
\bfzero_{m-p})_{\rO(m)}$ denote the $K$-type of~$\theta$ in
\eqref{eq6}. Let $\tau_0^l$ denote the isomorphic $K_{n,m}$-type in
$W_0$. We would like to construct a $K_{n,m}$-module isomorphism
$\phi_l : \tau^l \rightarrow \tau_0^l$ such that $\phi := \oplus_l
\phi_l : \theta \rightarrow W_0$ becomes an
$(\frakso(n+m,\bbC),K_{n,m})$-isomorphism. This would prove the
lemma. In order to achieve this, we will follow the method in
Proposition 8.2.1(i) in \cite{Lo}.  There we construct an
$(\frakso(n+m,\bbC),K_{n,m})$-isomorphism from $\theta$ to its
Hermitian dual $\theta^h$. The construction is valid if we replace
$\theta^h$ with $W_0$. A main ingredient of the proof is the claim
above that $W_0(\mu)$ and $\theta_p^{n,1}(L(\mu))$ are isomorphic
$(\frakso(n+1,\bbC), K_{n,1})$-modules. We refer the reader to Section
8.5 in \cite{Lo} for details.

Finally suppose $m < n$. Then we interchange the role of $m$ and $n$
and proceed as before. This proves the lemma.
\end{proof}

\subsection*{Remark}
If $m \geq n = 2p$, then the above lemma could be deduced from
\cite{T} where Trapa computes the characteristic cycles and
annihilator ideal of $\theta$. It is very likely that his
method could be extended to give a proof the lemma without the
assumption on the correspondences of infinitesimal characters.

\begin{proof}[Proof of Theorem \ref{T11}]
We have proved (i) in the paragraph after \eqref{eq9}. Next we prove~(ii). 
The condition $r' \geq 2p$ implies that $\theta_p^{n,r}(L(\mu))$
could have regular infinitesimal character. Let $\Gamma_0 =
\Gamma_{\frakg,K_{n,r,r'}}^{\frakg,K_{n+r,r'}}$ be the Zuckerman
functor. We set $W_0 = \Gamma_0^{pr}(\calF \thetaone)$
to be the left hand side of \eqref{eq9}.  By Proposition \ref{P32},
$(\Gamma_{\frakg_1,K_{n,r}}^{\frakg_1,K_{n+r}})^{pr}(
\theta_p^{n,r}(L(\mu))) = \rd_{n+r}^p \xi(\mu)_{\rO(n+r)}$.  Putting
this back into \eqref{eq9} and comparing it with \eqref{eq6}, we
conclude that $W_0$ has the same $K_{n+r,r'}$-types as
$\theta_p^{n+r,r'}(1)$. By the naturality of the Zuckerman functor, the
annihilator ideal of a $U(\frakso(n+m,\bbC))$-module $V$ will also
annihilate $\Gamma_0^i (V)$. Let $I_p^{n,m}$ denote the annihilator
ideal of $\thetaone$. Then $I_p^{n,m}$ annihilates $W_0$. In
particular the correspondence of infinitesimal characters in
Lemma~\ref{L31}(i) holds for $W_0$. By Lemma~\ref{L33}, $W_0 =
\theta_p^{n+r,r'}(1)$.  This proves~(ii).

Next we prove (iii). In the proof of (ii), we have shown that
$I_p^{n,m}$ annihilates $W_0 = \theta_p^{n+r,r'}(1)$, i.e. $I_p^{n,m}
\subseteq I_p^{n+r,r'}$. Similarly if we interchange the role of $n$
and $m$, then we get$(\Gamma_{\frakg,K_{n,r,r'}}^{\frakg,K_{n,m}})^{pr}(\calF
\theta_p^{n+r,r'}) = \thetaone$ which implies $I_p^{n,m}
\supseteq I_p^{n+r,r'}$.  This proves the first assertion of
(iii). The last assertion follows immediately from \cite{T} which
states that the annihilator ideal of $\theta_p^{2p,m+n-2p}(1)$ is the
maximal primitive ideal with infinitesimal character $\rho_{n+m}- p
\bfone_{\frac{n+m}{2}}$.

We will now prove (iv). 
The correspondence of infinitesimal characters is a result of the annihilator ideal so it is independent of the
real form. 
Hence we may assume that $a_1 = 0$, $a_2 = r$, $(b_1, b_2) = (n,r')$, and  $m = a_2 + b_2 = r + r'$.
In this case we have verified the correspondence in Lemma~\ref{L31}. This proves (iv).
This also  completes the proof of Theorem~\ref{T11}.
\end{proof}

\begin{proof}[Proof of Theorem \ref{T12}]
  We set $m = r + r'$. Applying
  $(\Gamma_{\frakg,K_{n,t,m-t}}^{\frakg,K_{n+t,m-t}})^{pt}$ to
  \eqref{eq7}, we get, as $(\frakg_1,K_{n,r})\times \rO(r')$-module, 
\[
(\Gamma_{\frakg,K_{n,t,m-t}}^{\frakg,K_{n+t,m-t}})^{pt}
(\calF \thetaone) = \bigoplus_{\mu \in \Lambda(\rO(r'))}
(\Gamma_{\frakg_1,K_{n,t,r-t}}^{\frakg_1,K_{n+t,r-t}})^{pt}(\calF
\theta_p^{n,r}(L(\mu))) \boxtimes \mu_{\rO(r')}.
\]
By Theorem \ref{T11}, the left hand side is
$\theta_p^{n+\rpp,m-\rpp}(1)$.  Theorem \ref{T12} follows by comparing
the above equation with \eqref{eq7} for $\theta_p^{n+\rpp,m-\rpp}(1)$.
\end{proof}

\subsection{Outside of stable range} \label{S34} Let $m = r + r'$. We
will discuss the situation when $(p,n,m)$ is outside of the stable
range. As in the stable range case, we will multiply the local theta
lifts by suitable genuine characters to get Harish-Chandra modules
$\thetaone$ and $\theta_p^{n,r}(L(\mu))$ of the linear groups
$\rO(n,m)$ and $\rO(n,r)$ respectively. We will also refer $\thetaone$
and $\theta_p^{n,r}(L(\mu))$ as the local theta lifts.

First we look at $\thetaone$. By symmetry, we assume that $n
\leq m$. Outside the stable range, $\thetaone$ is nonzero if and
only if
\begin{enumerate}[(I)]
\item $p \leq n \leq 2p-1$ and $m = n + 2$ or

\item $n = m \leq 2p$.
\end{enumerate}

First we consider (I). Let $\rd_n$ denote the one dimensional
character of $\rO(n,n+2)$ which is $\det_{\rO(n)}$ on $\rO(n)$ and
trivial on $\rO(n+2)$. Then by Lemma 3.5.1 in \cite{Lo}, $\rd_n
\theta_p^{n,n+2}(1)$ has the same $K_{n,n+2}$-types as those of
$\theta_{n-p}^{n,n+2}(1)$ in the stable range.  By Proposition~6.3.1 in
\cite{Lo}, the correspondence of the infinitesimal characters between
$\frakso(n+1,\bbC)$ and $\frakso(m-1,\bbC) = \frakso(n+1,\bbC)$ is
given by the identity map. Hence Lemma \ref{L33} gives $\rd_n
\theta_p^{n,n+2}(1) = \theta_{n-p}^{n,n+2}(1)$, and we are back to the
stable range. Equation \eqref{eq7} is valid for $\thetaone$ so
we also have $\theta_p^{n,r}(L(\mu)) = \theta_{n-p}^{n,r}(L(\mu))$ if
$m = n+2 = r+ r'$.

For (II), $\theta_p^{n,n}(1)$ is the dual representation of
$\theta_{n-1-p}^{n,n}(1)$ in the stable range \cite{Z}.  Since all the
$K_{n,r}$-types are self dual and $-1$ belongs to the Weyl group,
Lemma \ref{L31}(i) holds for $\theta_p^{n,n}(1)$. By Lemma \ref{L33},
$\theta_p^{n,n}(1) = \theta_{n-1-p}^{n,n}(1)$, and we are back to the stable
range.

For all the cases considered so far, we could reduce the computation
to the stable range. We will leave it to the readers to formulate the
corresponding Theorems~\ref{T11} and \ref{T12}.

Unfortunately in Case (II), \eqref{eq7} is no longer valid. It is only
true if we replace $\thetaone$ and $\theta_p^{n,r}(L(\mu))$ by the
maximal Howe quotients $\Theta_p^{n,m}(1)$ and $\Theta_p^{n,r}(L(\mu))$
respectively. We refer to \cite{Lo} for the definition of maximal Howe
quotients. Then $\thetaone$ and $\theta_p^{n,r}(L(\mu))$ are
the unique irreducible quotients of $\Theta_p^{n,m}(1)$ and
$\Theta_p^{n,r}(L(\mu))$ respectively. This complicates the investigation
of the transfers of $K_{n,r}$-types for $\theta_p^{n,r}(L(\mu))$. One
could manage by a careful analysis of the submodules and the $K$-types
of $\Theta_p^{n,m}(1)$ and $\Theta_p^{n,r}(L(\mu))$ as given in
\cite{Lo}. However it is tedious so we will omit this case.

\section{Cohomologically induced modules} \label{S4}

The objective of this section is to setup some notation for computing
$A(\lambda)$ and $\overline{A}(\lambda)$ in Theorem \ref{T14}. 
We will follow the notation in \cite{KV}.

From now on, we assume that $2p \leq n \leq m$, $2p < m$ and $m = r +
r'$.  We use a subscript 0 to denote a real Lie algebra. Those without
are complex Lie algebras. Let $\frakg_1 = \frakso(n+r,\bbC)$. In
Section \ref{S13}, we have chosen a compact Cartan subalgebra
$\frakt_0$ of $\frakso(n,\bbR) \oplus \frakso(r,\bbR)$ in
$\frakso(n,r)$. Let $\frakt = \frakt_0 \otimes \bbC$ and $t = \dim
\frakt = [\frac{n}{2}] + [\frac{r}{2}]$. We choose a positive root
system $\Phi^+(\frakt)$ with respect to $\frakt$ such that the
positive roots $\varepsilon_i \pm \varepsilon_j$ for $1 \leq i < j
\leq [\frac{n}{2}]$ belongs to $\frakso(n)$, and $\varepsilon_i \pm
\varepsilon_j$ for $[\frac{n}{2}]+1 \leq i < j \leq [\frac{n}{2}] +
[\frac{r}{2}]$ belongs to~$\frakso(r)$.

Let $\lambda_0 = (\bfone_p, \bfzero_{t-p})$ and $\lambda_0' =
(p,p-1,\ldots,1, \bfzero_{t-p}) \in \sqrt{-1}\frakt_{0}^{\ast}$. Let
$\frakq =\frakl \oplus \frakn$ and $\frakq' =\frakl' \oplus \frakn'$
be the $\theta$-stable parabolic subalgebras in $\frakg_1$ defined by
$\lambda_0$ and $\lambda_0'$ respectively. We have $\frakl_0 =
\fraku(p) \oplus \frakso(n-2p,r)$ and $\frakl_0' = \fraku(1)^p \oplus
\frakso(n-2p,r)$. Let
\[
L = \rU(p)\times \rO(n-2p,r) \mbox{ and } L' = \rU(1)^p \times
\rO(n-2p,r)
\]
be the subgroups in $\rO(n,r)$ whose Lie algebras are $\frakl_0$ and
$\frakl_0'$ respectively. Let $M = \rU(p) \times K_{n-2p,r}$ and $M' =
\rU(1)^p \times K_{n-2p,r}$ be maximal compact subgroups of $L$ and $L'$
respectively.  We note that $\frakq'$, $L'$ and $M'$ are the same as those
in Section \ref{S13}.

\medskip

Under the adjoint action of $L$, the radical $\frakn$ of $\frakq$
decomposes as
\[
\frakn = \wedge^2(\bbC^p) \oplus (\bbC^p \otimes \bbC^{n-2p+r})
\]
where $\bbC^p$ is the standard representation of $\rU(p)$ and
$\bbC^{n-2p+r}$ is the standard representation of $\rO(n-2p,r)$. Hence
\begin{equation}\label{eq14}
\wedge^{\mathrm{top}}\frakn \cong \det{}_{\rU(p)}^{n+r-p-1}
 \otimes\det{}_{\rO(n-2p,r)}^p.
\end{equation}

Let $\mu = (\mu_1, \ldots, \mu_{m-r})$ be a highest weight of $\rO(m-r)$
such that $\mu_i = 0$ for $i > p$. We set $Z = Z(\lambda)$ to be the
irreducible finite dimensional $\rU(p)$-module of highest weight
\begin{equation} \label{eq15}
\lambda = \lambda_\mu = (\mu_1, \ldots, \mu_p) + \frac{m-n-2r}{2}
\bfone_p.
\end{equation}

We let $\rO(n-2p,r)$ act on it by $\det_{\rO(n-2p,r)}^p$. We continue
to denote its highest $\frakt$-weight by~$\lambda$.  Let
$\overline{\frakq}$ denote the parabolic subalgebra that is opposite
to $\frakq$. We put $Z^\sharp = Z \otimes \wedge^{\mathrm{top}}
\frakn$. Let $M = L\cap K_{n,r}\cong \rU(p)\times
  \rO(n-2p)\times \rO(r)$. We extend $Z^\sharp$ to a $(\overline{\frakq},M)$-module 
where $\overline{\frakn}$ acts trivially. We define the $(\frakg_1, M)$-module
\[
\ind_{\bar{\frakq}, M}^{\frakg_1, M} Z^\sharp = U(\frakg_1)
\otimes_{\bar{\frakq}} Z^\sharp.
\]
We will write $\ind Z^\sharp$ if there is no fear of confusion.

We set $s_0 = \dim (\frakn \cap \frakk) = p(n-2p) + \frac{p(p-1)}{2}$.
Note that $s_0$ is independent of $r$. Let $s_1 = \frac{p(p-1)}{2}$ be
the dimension of a maximal nilpotent subalgebra of $\frakgl(p)$.  Let
$K = K_{n,r}$ and let $(\Pi_{\frakg_1,M}^{\frakg_1,K})_j$ be the
$j$-th derived functor of the Bernstein functor. Let $\bbC_\lambda$ be
a character of~$L'$ as in Section \ref{S13} and let
$\bbC_\lambda^\sharp = \bbC_\lambda \otimes \wedge^{\mathrm{top}}
\frakn'$. We recall $A(\lambda)$ in Section \ref{S13}. We claim that
\begin{equation} \label{eq16}
A(\lambda) = (\Pi_{\frakg_1,M'}^{\frakg_1,K})_{s_0+s_1}
(\ind_{\overline{\frakq'},M'}^{\frakg_1,M'} \bbC_\lambda^\sharp) =
(\Pi_{\frakg_1,M}^{\frakg_1,K})_{s_0}(\ind Z^\sharp).
\end{equation}
The claim follows from the Borel-Weil-Bott-Kostant theorem and a
standard spectral sequence argument. For example see Corollary
11.86(a) on page 683 in~\cite{KV}.

Instead of working with Bernstein functor modules, we prefer to use
Zuckerman functor modules. By the Hard Duality Theorem 3.5 in
\cite{KV}, we have
\begin{equation} \label{eq17}
A(\lambda) = (\Gamma_{\frakg_1,M}^{\frakg_1,K})^{s_0} \ind Z^{\sharp}.
\end{equation}
The module $A(\lambda)$ has infinitesimal character
$\lambda + \rho(\frakg_1)$.

\begin{lemma} \label{L41}
  Let $\lambda = \lambda_\mu = \mu + \frac{m-n-2r}{2} \bfone_p$, $Z$
  and $A(\lambda)$ as above.
\begin{enumerate}[(i)]
\item The module $A(\lambda)$ is $\rO(n)$-admissible. An $\rO(n)$-type
  of $A(\lambda)$ has highest weight
\[
\rd_n^p \left(\mu_1 + \kappa_1 + \frac{m-n}{2}, \ldots, \mu_p +
  \kappa_p + \frac{m-n}{2}, \bfzero_{n-p} \right)
\]
where $\kappa_i$ are non-negative integers.

\item The module $A(\lambda)$ contains the $K_{n,r}$-type $\tau_{\min}
  = \rd_n^p (\mu + \frac{m-n}{2} \bfone_p)_{\rO(n)} \otimes
  \bbC_{\rO(m)}$ with multiplicity one. It is the minimal
  $K_{n,r}$-type. It is also the image of the bottom layer map.
\end{enumerate}
\end{lemma}
In (ii), we refer to page 642 in \cite{KV} for the
  definition of minimal $K$-types for disconnected~$K$.

We postpone the proof of the above lemma to Section \ref{S54}.
Alternatively one could also verify this lemma directly using the
Blattner's formula (see Theorem 5.64 in \cite{KV}).  The fact that
$A(\lambda)$ is admissible with respect to $\SO(n)$ also follows from
a very general criterion in~\cite{Ko}. 

By (ii), $\overline{A}(\lambda)$ is the irreducible
$(\frakg_1,K_{n,r})$-subquotient of $A(\lambda)$ generated by the
minimal $K_{n,r}$-type $\tau_{\min}$ in (ii).

\medskip

Let $K = K_{n,r}$ and $K^1 = M K_{n,r}^0$. If $n > 2p$
then $K^1 = K$.  If $n = 2p$ then $K^1 = \SO(n)
\times \rO(r)$ is a subgroup of index two in $K$.  Let
$\calC(\frakg_1, K)$ denote the category of $(\frakg_1,
K)$-modules etc... For $V \in \calC(\frakg_1,K^1)$, we
define
\[
{\mathrm{induced}}_{\frakg_1,K^1}^{\frakg_1,K} V = \{ f : K
\rightarrow V : f(k_1k) = k_1 f(k) \mbox{ for all } k_1 \in K^1, k \in
K \}.
\]
Applying Proposition 2.77 in \cite{KV} to \eqref{eq17} gives 
\begin{equation} \label{eq18}
A(\lambda) = {\mathrm{induced}}_{\frakg_1,K^1}^{\frakg_1,K}
(\Gamma_{\frakg_1,M}^{\frakg_1,K^1})^{s_0}(\ind Z^\sharp)
\end{equation}
where $(\Gamma_{\frakg_1,M}^{\frakg_1,K^1})^{s_0}$ is the
$s_0$-th derived functor of the Zuckerman functor. Let
$\calC(\frakg_1, K)$ denote the category of $(\frakg_1,
K)$-modules etc...  The above equation shows that we could
compute $A(\lambda)$ first in $\calC(\frakg_1, K^1)$ and then
apply the exact functor
${\mathrm{induced}}_{\frakg_1,K^1}^{\frakg_1,K}$.

Given a module $V$ in $\calC(\frakg_1,K)$, we let
$\calF_{\frakg_1,K}^{\frakg_1,K^0} V \in \calC(\frakg_1,K^0)$ denote
the restriction of $V$ as a $(\frakg_1,K^0)$-module. We recall the
$(\frakg_1,K^0)$-module $A_{\frakq'}(\lambda) =
(P_{\overline{\frakq'},(M')^0}^{\frakg,K^0})_{s_0'}
\bbC_\lambda^\sharp$ on page 330 in \cite{KV}.

\begin{prop} \label{P42}
We have $A_{\frakq'}(\lambda) =
\calF_{\frakg_1,K^1}^{\frakg_1,K^0}
(\Gamma_{\frakg_1,M}^{\frakg_1,K^1})^{s_0}(\ind Z^\sharp)
\subseteq \calF_{\frakg_1,K}^{\frakg_1,K^0} A(\lambda)$.
\end{prop}

If $n > 2p$, then $K = K^1$ and $A_{\frakq'}(\lambda) =
\calF_{\frakg_1,K}^{\frakg_1,K^0} A(\lambda)$.

\begin{proof}
The inclusion on the right follows from \eqref{eq18}.  By the
spectral sequence and the Hard Duality Theorem (cf. \eqref{eq16} and
\eqref{eq17}), we have
\begin{equation} \label{eq19}
A_{\frakq'}(\lambda) = (\Pi_{\frakg_1,M^0}^{\frakg_1,K^0})_{s_0}(
\ind_{\bar{\frakq}, M^0}^{\frakg_1, M^0}
Z^\sharp)
= (\Gamma_{\frakg_1,M^0}^{\frakg_1,K^0})^{s_0}(
\ind_{\bar{\frakq}, M^0}^{\frakg_1, M^0} Z^\sharp).
\end{equation}
Note that $M \cap K^0$ is the identity connected component $M^0$ of
$M$.  By Section 6.2 of \cite{Vo} or Section I.8.9 in \cite{BW},
\begin{equation} \label{eq20}
\calF_{\frakg_1,K^1}^{\frakg_1,K^0} \circ
(\Gamma_{\frakg_1,M}^{\frakg_1,K^1} )^{s_0}=
(\Gamma_{\frakg_1,M^0}^{\frakg_1,K^0})^{s_0} \circ
\calF_{\frakg_1,M}^{\frakg_1,M^0}.
\end{equation}
Applying this functor to $\ind Z^\sharp$ and comparing it with
\eqref{eq19} gives the first equality of the proposition.
\end{proof}

\subsection{} \label{S43} Although the group $M$ is different from
$K_{n,r,m-r}$, Section~\ref{S22} applies to
$\Gamma_{\frakg_1,M}^{\frakg_1,K} V$. More precisely, let $V \in
\calC(\frakg_1,M)$.  We define $(\Gamma_M^K)^i V = H^i(\frakk,M; V
\otimes \calH(K))$ where $\calH(K)$ denotes the bi-$K$-finite
continuous functions on~$K$. Then $(\Gamma_M^K)^i V$ has a natural
$(\frakg_1,K)$-module structure and $(\Gamma_M^K)^i V =
(\Gamma_{\frakg_1,M}^{\frakg_1,K})^i V$ in $\calC(\frakg_1,K)$. Since $K/M
\simeq \rO(n)/(\rU(p) \times \rO(n-2p))$, one could naturally identify
$(\Gamma_M^K)^i \simeq (\Gamma_{\rU(p) \times
  \rO(n-2p)}^{\rO(n)})^i$. This allows us to compute
$(\Gamma_{\frakg_1,M}^{\frakg_1,K})^i V$ within the category
$\calC(\frakso(n,\bbC), \rU(p) \times \rO(n-2p))$.

\section{Proof of Theorem \ref{T14}} \label{S5}

This section contains the proofs of Lemma \ref{L41} and Theorem
\ref{T14}. The initial argument leading to Lemma \ref{L55} follows
Section 4 in \cite{LS} so we will omit details and refer the reader
to that paper.  Unfortunately we cannot totally dispense with it
because we will need them later in the proof of Theorem \ref{T14}.

\medskip

We set $\frakg = \frakso(n+m,\bbC)$ and $\frakg_1 =
\frakso(n+r,\bbC)$.  We recall the theta stable parabolic subalgebra
$\frakq = \frakl \oplus \frakn$ of $\frakg$ where $\frakl =
\frakgl(p,\bbC) \oplus \frakso(n+m-2p,\bbC)$. We define $\frakq_1 =
\frakq \cap \frakso(n+r,\bbC) = \frakl_1 \oplus \frakn_1$. We have a
decomposition $\frakn=\frakn_1 \oplus \frakn_2$ such that $\frakn_2 =
\bbC^p \otimes \bbC^{m-r}$ is a tensor product of standard
representations of $\rU(p)$ and $\rO(m-r)$, while the group
$\rO(n-2p,r)$ acts trivially on it. We extend $\frakn_2$ to a
representation of $\rU(p) \times \rU(m-r)$. It is well known that (see
\cite{GoW} and \cite{H3})
\[
{\mathrm{Sym}}^N \frakn_2 = \sum_{\mu} \mu_{\rU(p)} \otimes
\mu_{\rU(m-r)}
\]
where the sum is taken over all partitions $\mu$ of $N$ such that the
length of $\mu$ is not longer than $\min(p,m-r)$.  We restrict the
summand $\mu_{\rU(m-r)}$ to $\rO(m-r)$ and we denote it by
$\mu_{\rU(m-r)}|_{\rO(m-r)}$.

With reference to \eqref{eq14} and \eqref{eq15}, let $\lambda =
-\frac{n+m}{2} \bfone_p$, $Z = \det_{\rU(p)}^{-\frac{n+m}{2}} \otimes
\det_{\rO(n-2p,m)}^p$ and
\[
Z^{\sharp} = \det{}_{\rU(p)}^{\frac{n+m}{2}-p-1}\otimes
\bbC_{\rO(n-2p,m)}.
\]
Let $\symm : \Sym(\frakg) \rightarrow U(\frakg)$ denote the
symmetrization map (see \S 0.4.2 in \cite{Wa1}). Let $M = \rU(p)
\times K_{n-2p,m}$.  By the Poincare-Birkhoff-Witt theorem,
\begin{equation} \label{eq21} \ind Z^\sharp =
  \ind_{\overline{\frakq},M}^{\frakg,M} Z^\sharp= U(\frakg)
  \otimes_{\bar{\frakq}} Z^\sharp = U(\frakn_1) \otimes
  \symm(\Sym(\frakn_2)) \otimes Z^\sharp
\end{equation}
as $\rU(p) \times K_{n-2p,r} \times \rO(m-r)$-module.  We define $F_N$
to be the $\frakg_1$-submodule of $\ind Z^\sharp$ generated by $1
\otimes \symm(S_N(\frakn_2)) \otimes Z^\sharp$. Hence $\{ F_N : N = 0,
1, 2, \ldots \}$ forms an exhaustive increasing filtration of
$(\frakg_1, \rU(p) \times K_{n-2p,r}) \times \rO(m-r)$-submodules of
$\ind Z^\sharp$.  We set
\[
V_r(\mu)  = \ind_{(\bar{\frakq}_1, \rU(p) \times
  K_{n-2p,r})}^{(\frakg_1, \rU(p) \times
  K_{n-2p,r})} \calF_{\frakl_1, \rU(p)\times
    K_{n-2p,r}}^{\bar{\frakq}_1,\rU(p)\times K_{n-2q,r}}  \left(
(\mu + (\frac{m+n}{2}-2p)\bfone_p)_{\rU(p)} \otimes \bbC_{\rO(n-2p,r)} \right).
\]

Let $\calB$ (resp. $\calB(N)$) denote the partitions (resp. partitions
of $N$) of length not more than $\min(p,m-r)$.  We will now state a
special case of a known fact which is used in proof of the Blattner's
formula in \cite{KV}.
\begin{lemma} \label{L51}
  For every positive integer $N$, we have an isomorphism of
  \\ $(\frakg_1,\rU(p) \times K_{n-2p,r}) \times
  \rO(m-r)$-modules
\[
F_N/F_{N-1} = \bigoplus_{\mu \in \calB(N)} V_r(\mu) \otimes
(\mu_{\rU(m-r)} |_{\rO(m-r)}). \ \qed
\]
\end{lemma}

\smallskip

\noindent \underline{Case 1.}  We first consider the case $r = 0$,
$\frakg_1=\frakso_n(\bbC)$ and $V_r(\mu) = V_0(\mu)$.  By an
irreducibility criterion in Theorem 9.12 in \cite{Hu}, for $\mu \in
\calB$, $V_0(\mu)$ is an irreducible generalized Verma module of
$\frakso(n)$. The infinitesimal character of $V_0(\mu)$ is the same as
the infinitesimal of
$(\mu+(\frac{m-n}{2})\bfone_p)_{\rO(n)}$. Hence these
infinitesimal characters are pairwise different for different
partitions $\mu$. It follows that the filtration $F_N$ splits and
\eqref{eq21} becomes
\begin{equation}
\ind Z^\sharp = \bigoplus_{\mu \in \calB}  V_0(\mu)
\otimes (\mu_{\rU(m)}|_{\rO(m)}) \label{eq22}
\end{equation}
where the sum is taken over all partitions $\mu$ of length no more
than $\min(p,m)$. Let $\Gamma =
\Gamma_M^{K_{n,m}} = \Gamma_{\rU(p) \times \rO(n-2p)}^{\rO(n)}$ as in
Section \ref{S43}. Applying its derived functor $\Gamma^j$ to
\eqref{eq22} gives
\[
\Gamma^j(\ind Z^\sharp) = \bigoplus_{\mu \in \calB}
\Gamma^j (V_0(\mu)) \otimes (\mu_{\rU(m)}|_{\rO(m)}).
\]
where we recall $\lambda = -\frac{n+m}{2} \bfone_p$.  Since $s_0 =
\dim(\frakn_1)$ in this case, by the Borel-Weil-Bott-Kostant theorem,
$\Gamma^j(V_0(\mu)) = 0$ if $j \neq s_0$ and $\Gamma^{s_0}(V_0(\mu))
= \rd_n^p (\mu + \frac{m-n}{2} \bfone_p)_{\rO(n)}$. When $j = s_0$,
the above equation becomes
\begin{equation} \label{eq23}
A(\lambda) = \bigoplus_{\mu \in \calB}  \rd_n^p (\mu +
  \frac{m-n}{2} \bfone_p)_{\rO(n)} \otimes (\mu_{\rU(m)}|_{\rO(m)})
\end{equation}
and it has minimal $K_{n,m}$-type $\tau_{\min} = \rd_n^p
(\frac{m-n}{2} \bfone_p)_{\rO(n)} \otimes \bbC_{\rO(m)}$ which
occurs with multiplicity one.

\begin{lemma} \label{L52}
  Suppose $W_1, W_2, W_3$ are $(\frakso(n),\rU(p) \times
  \rO(n-2p))$-subquotients of $\ind Z^\sharp$ and suppose they satisfy
  an exact sequence
\[
0 \rightarrow W_1 \rightarrow W_2 \rightarrow W_3 \rightarrow 0.
\]
\begin{enumerate}[(i)]
\item Each $W_k $ is a direct summand of \eqref{eq22} and $W_k =
  \bigoplus_{\mu'} V_0(\mu')$ is a direct sum taken over certain
  partitions $\mu'$ of length no more than $\min(p,m)$.
  
\item If $j \neq s_0$, then $\Gamma^j W_k = 0$. Moreover we have an
  exact sequence
\[
0 \rightarrow \Gamma^{s_0} W_1 \rightarrow \Gamma^{s_0} W_2
\rightarrow \Gamma^{s_0} W_3 \rightarrow 0.
\]

\item Each $\Gamma^{s_0} W_j$ is $\rO(n)$-admissible.
\end{enumerate}
\end{lemma}

\begin{proof} 
  Part (i) follows from \eqref{eq22}.  Part (ii) follows from (i) and
  the Borel-Weil-Bott-Kostant theorem alluded above. By (ii) each
  $\Gamma^{s_0} W_j$ is an $\rO(n)$-subquotient of $A(\lambda)$ in
  \eqref{eq23} which is $\rO(n)$-admissible. This proves (iii).
\end{proof}

\medskip

\noindent \underline{Case 2.}  Now we return to the general $r$ for
$\frakg_1=\frakso(n+r,\bbC)$.  Let $V_r(\mu)$ be a subquotient in Lemma
\ref{L51}.  By Lemma \ref{L52}(i), $V_r(\mu) = \bigoplus_{\mu'}
V_0(\mu')$.  

\begin{lemma} \label{L53}
If $V_0(\mu')\subseteq V_r(\mu)$, then $\mu' = \mu + \kappa$
where $\kappa$ is a $p$-tuple of non-negative integers.
Furthermore $V_r(\mu)$ contains $V_0(\mu)$ with multiplicity one and
$\rO(r)$ acts trivially on it. 
\end{lemma}

The proof consists of setting up a filtration of 
$(\frakso(n,\bbC), \rU(p) \times K_{n-2p})$-submodules of~$V_r(\mu)$ 
in the same fashion as Lemma \ref{L51} and study the graded module.  
We refer to Lemma 4.5 in \cite{LS} for a detailed proof of a similar result.

\subsection{} \label{S54} {\it Proof of Lemma~\ref{L41}.}  By Lemma
\ref{L53} above, $V_r(\mu) = \bigoplus_{\kappa} V_0(\mu + \kappa)$
where the sum~$\kappa$ is taken with multiplicities over a set of
$p$-tuples of non-negative integers. Let $\lambda = -\frac{n+m}{2}
\bfone_p$ and $\lambda_\mu = \mu + (\frac{m-n-2r}{2})
\bfone_p$. Applying $\Gamma^{s_0} = (\Gamma_{\rU(p) \times
  K_{n-2p,r}}^{K_{n,r}})^{s_0} = (\Gamma_{\rU(p) \times
  \rO(n-2p)}^{\rO(n)})^{s_0}$ to $V_r(\mu)$ gives
\begin{equation} \label{eq24}
A(\lambda_\mu) = \Gamma^{s_0}(V_r(\mu)) = \bigoplus_{\kappa}  \Gamma^{s_0}
V_0(\mu + \kappa) = \bigoplus_{\kappa}  \rd_n^p (\mu + \kappa +
\frac{m-n}{2} \bfone_p)_{\rO(n)}.
\end{equation}
It is $\rO(n)$-admissible by Lemma \ref{L52}(iii). Furthermore by the
second assertion of Lemma~\ref{L53}, $A(\lambda_\mu)$ above contains
$\rd_n^p (\mu + \frac{m-n}{2} \bfone_p)_{\rO(n)} \otimes
\bbC_{\rO(r)}$ with multiplicity one. Clearly this is the minimal
$K_{n,r}$-type of $A(\lambda_\mu)$. This proves
Lemma~\ref{L41}. \qed

\smallskip

Applying Lemma \ref{L52} to the filtration $F_N$ gives the next
lemma.

\begin{lemma} \label{L55}
In the category of $(\frakg_1,K_{n,r}) \times \rO(m-r)$-modules,
  $\Gamma^{s_0}( F_N)$ is an exhaustive increasing filtration of
  $A(\lambda)$ and
\[
\Gamma^{s_0} (F_N) / \Gamma^{s_0}(F_{N-1}) = \Gamma^{s_0}
(F_N / F_{N-1}) = \bigoplus_{\mu \in \calB(N)} A(\lambda_\mu)
\otimes (\mu_{\rU(m-r)}|_{\rO(m-r)}). \ \qed
\]
\end{lemma}

\begin{proof}[Proof of Theorem \ref{T14}(i)] 
  Let $c = \frac{n+m}{2}$, $\lambda = - c \bfone_p$ and $\frakg =
  \frakso(c,\bbC)$. We will denote the $(\frakg, K_{n,m})$-module
  $\overline{A}(\lambda)$ by $\overline{A}_{n,m}(\lambda)$.  Let
  $M_{n,m} = \rU(p) \times K_{n-2p,m}$ and let $\Gamma_{n,m}^i$ denote
  the derived Zuckerman functor $(\Gamma_{M_{n,m}}^{K_{n,m}})^i$ as in
  Section~\ref{S43}.
  
  First we consider the special case when $n=2p < m=2c-2p$ and $s_0 =
  p(p-1)/2$.  Let $\ind Z^\sharp$ denote $\ind_{\bar{\frakq},\rU(p)
    \times \rO(2c-2p)}^{\frakg,\rU(p) \times \rO(2c-2p)} Z^\sharp$
  where $\frakq = (\fraku(p) \oplus \frakso(2c-2p)) \oplus \frakn$.
  In \cite{Kn}, Knapp constructs an exact sequence of $\frakg$-modules
\begin{equation} \label{eq25}
0 \rightarrow Q' \rightarrow \ind Z^\sharp \rightarrow Q \rightarrow 0.
\end{equation}
 We
check that the above extends to an exact sequence of $(\frakg, \rU(p)
\times \rO(2c-2p))$-modules. By Lemma \ref{L52}(i) $\ind
Z^\sharp$, $Q$ and $Q'$ are direct sums of $V_0(\mu)$'s. Furthermore
$Q$ contains the bottom layer $V_0(\mu = 0)$.  Lemma~\ref{L52}(ii)
gives
\begin{equation} \label{eq26}
0 \rightarrow \Gamma_{2p,2c-2p}^{s_0} Q' \rightarrow
\Gamma_{2p,2c-2p}^{s_0} \ind Z^\sharp \rightarrow
\Gamma_{2p,2c-2p}^{s_0} Q \rightarrow 0.
\end{equation}
Since $Q$ contains $V_0(0)$, $\Gamma_{2p,2c-2p}^{s_0} Q$ contains
the minimal $K_{2p,2c-2p}$-type of $A_{2p,2c-2p}(\lambda) =
\Gamma_{2p,2c-2p}^{s_0} \ind Z^\sharp$.
Hence $\overline{A}_{2p,2c-2p}(\lambda)$ is a subquotient of
$\Gamma_{2p,2c-2p}^{s_0} Q$.

Let $\Gamma_0 = \Gamma_{M_{2p,2c-2p}^0}^{K_{2p,2c-2p}^0}$,
$\Gamma_1 = \Gamma_{M_{2p,2c-2p}}^{K_{2p,2c-2p}^1}$ and $\pi_0 =
\Gamma_1^{s_0} Q$. By \eqref{eq20},
$\calF_{\frakg,K_{2p,2c-2p}^1}^{\frakg,K_{2p,2c-2p}^0} \pi_0 = \Gamma_0^{s_0}
  Q$. Knapp shows that $\pi_0$ is a unitarizable
  $(\frakg,K_{2p,2c-2p}^1)$-quotient of $\Gamma_1^{s_0} \ind Z^\sharp$
  and $\pi_0$ contains the minimal $K_{2p,2c-2p}^1$-type of
  $\Gamma_1^{s_0} \ind Z^\sharp$. Trapa proves that $\pi_0$ is an
  irreducible module \cite{T}. We remark that Knapp and Trapa uses
  $\Gamma_0$ instead of $\Gamma_1$ but their proofs could be easily
  adapted to $\Gamma_1$ with the help of \eqref{eq20}.

Let $\tilde{\pi}_0 =
{\mathrm{induced}}_{\frakg,K_{2p,2c-2p}^1}^{\frakg,K_{2p,2c-2p}}
\pi_0$. By \eqref{eq18}, $\Gamma_{2p,2c-2p}^{s_0} Q = \tilde{\pi}_0$
and $\overline{A}_{2p,2c-2p}(\lambda)$ is its $(\frakg,
K_{2n,2c-2p})$-subquotient.  Using the description of
$K_{2p,2c-2p}^0$-types of $\pi_0$ in \cite{Kn}, one checks that
$\tilde{\pi}_0$ has the same $K_{2p,2c-2p}$-types as
$\theta_p^{2p,2c-2p}(1)$. Hence $\tilde{\pi}_0 =
\theta_p^{2p,2c-2p}(1)$ by Lemma~\ref{L33}. In particular
$\tilde{\pi}_0$ is irreducible and therefore it is equal to
$\overline{A}_{2p,2c-2p}(\lambda)$. This proves that
$\overline{A}_{2p,2c-2p}(\lambda) = \theta_p^{2p,2c-2p}(1)$.
Alternatively, this also follows from \cite{PT} where Paul and Trapa
prove that $\pi_0$ is a $(\frakg,K_{2p,2c-2p}^0)$-submodule
of~$\theta_p^{2p,2c-2p}(1)$.

Now we turn to the general case, i.e. $2p < n$ and
$s_0 = p(n-2p) + \frac{p(p-1)}{2}$.  We consider~\eqref{eq25} as an
exact sequence of $(\frakg, \rU(p) \times
K_{n-2p,m})$-modules. Repeating the argument with $n$ instead of $2p$
in \eqref{eq26} gives
\[
0 \rightarrow \Gamma_{n,m}^{s_0} Q' \rightarrow \Gamma_{n,m}^{s_0}
\ind Z^\sharp \rightarrow \Gamma_{n,m}^{s_0} Q
\rightarrow 0.
\]
By Section \ref{S43}, $\Gamma_{n,m}$ is essentially the functor
$\Gamma_1 := \Gamma_{\rU(p) \times \rO(n-2p)}^{\rO(n)}$.  Let
$\Gamma_2 = \Gamma_{\rO(2p) \times \rO(n-2p)}^{\rO(n)}$ and $\Gamma_3
= \Gamma_{\rU(p) \times \rO(n-2p)}^{\rO(2p) \times \rO(n-2p)}$.  By
Proposition 6.2.17 in \cite{Vo}, we have a spectral sequence
$\Gamma_2^a \Gamma_3^b \Rightarrow \Gamma_1^{a+b}$. By Lemma
\ref{L52}(ii), $\Gamma_3^b Q = 0$ if $b \neq p(p-1)/2$. Hence the
spectral sequence gives $\Gamma_{n,m}^{s_0} Q = \Gamma_1^{s_0} Q =
\Gamma_2^{p(n-2p)} \Gamma_3^{p(p-1)/2} Q = \Gamma_2^{p(n-2p)}
\theta_p^{2p,2c-2p}(1) = \thetaone$. The last equality follows from
Theorem~\ref{T11}(ii). Hence $\thetaone$ is an irreducible quotient of
$A_{n,m}(\lambda) = \Gamma_{n,m}^{s_0} \ind Z^\sharp$ containing the
minimal $K_{n,m}$-type, i.e.  $\overline{A}_{n,m}(\lambda) =
\thetaone$.  This proves Theorem~\ref{T14}(i).
\end{proof}

\begin{proof}[Proof of Theorem \ref{T14}(ii)]
  Since $\thetaone$ is an irreducible subquotient of $A(\lambda)$, it
  follows that $\theta_p^{n,r}(L(\mu)) \otimes \mu_{\rO(m-r)}$ is an
  irreducible subquotient of $A(\lambda)$, considered as
  $(\frakso(n+r,\bbC),K_{n,r}) \times \rO(m-r)$-module. This implies
  that $\theta_p^{n,r}(L(\mu))$ is an irreducible subquotient of
  $A(\lambda_{\mu'}) = \Gamma^{s_{0}}(V_r(\mu'))$ as in \eqref{eq24}
  for some $\mu'$ such that $(\mu')_{\rU(m-r)}$ contains
  $\mu_{\rO(m-r)}$.  We recall that $\thetaone$ contains the
  $K_{n,r}$-type $W = \rd_n^p (\mu +\frac{m-n}{2}\bfone_{m})_{\rO(n)}
  \otimes \bbC_{\rO(r)} = \Gamma^{s_0}(V_0(\mu))$. Hence
  $A(\lambda_{\mu'})$ contains $W$.

We claim that $\mu = \mu'$. Indeed we embed the torus $\frakt$ of
$\rO(m-r)$ in the torus of $\rU(m-r)$ so that the restriction of
positive $\rU(m-r)$ roots to $\rO(m-r)$ remains positive. By
Lemma~\ref{L53} and \eqref{eq24}, $\mu = \mu' + \kappa = \mu' +
(\kappa_1, \ldots, \kappa_p) \in \frakt^*$ where $\kappa_i \geq 0$. On
the other hand any weight of $(\mu')_{\rU(m-r)}$ is of the form
$\mu'-(\sum \alpha)$ where $\sum \alpha$ is a sum of positive roots of
$\rU(m-r)$ with multiplicities.  Hence $\mu' + \kappa = \mu =
\mu' - (\sum \alpha)$ on $\frakt$. This forces $\kappa_1= \ldots
=\kappa_p=0$ and $\mu=\mu'$.  This proves our claim.

  From our claim $\theta_p^{n,r}(L(\mu))$ is the irreducible subquotient of
  $A(\lambda_\mu)$ generated by $K_{n,r}$-type~$W$. Since $W$ is the
  unique minimal $K_{n,r}$-type of $A(\lambda_\mu)$, $\theta_p^{n,r}(L(\mu)) =
  \overline{A}(\lambda_\mu)$.  This proves Theorem~\ref{T14}(ii).
\end{proof}


\begin{thebibliography}{99}
\bibitem[BW]{BW} A. Borel and N. Wallach, {\em Continuous cohomology,
  discrete subgroups, and representations of reductive groups,} Annals
  of Mathematics Studies, vol. 94, Princeton University Press,
  Princeton, N.J., 1980.

\bibitem[EHW]{EHW} T. Enright, R. Howe and N. Wallach, {\em A
    classification of unitary highest weight modules.}  Representation
  theory of reductive groups (Park City, Utah, 1982), 97--143, Progr.
  Math., 40, Birkh\"{a}user Boston, Boston, MA, 1983.

\bibitem[GoW]{GoW} Roe Goodman and Nolan R. Wallach, {\em
    Representations and Invariants of the Classical Groups} Cambridge
  U. Press, 1998; third corrected printing, 2003.
    
\bibitem[GW]{GW} B.H. Gross and N.R. Wallach, {\em On quaternionic
    discrete series representations, and their continuations.} J.
  Reine Angew.  Math. 481 (1996), 73-123


\bibitem[H1]{H1} R. Howe, {\em Remarks on classical invariant
      theory.} Trans. Amer. Math. Soc. {\bf 313} (1989), no. 2,
    539--570.

\bibitem[H2]{H2} R. Howe {\em Transcending classical invariant theory.}
J. AMS {\bf 2} no. 3 (1989).

\bibitem[H3]{H3} Roger Howe {\em Perspectives on invariant theory:
    Schur duality, multiplicity-free actions and beyond.}  The Schur
    lectures (1992) (Tel Aviv), 1--182, Israel Math. Conf. Proc., 8,
    Bar-Ilan Univ., Ramat Gan, 1995.

\bibitem[HZ]{HZ} C.-B. Zhu and J.-S. Huang, {\em On certain small
representations of indefinite orthogonal groups.} Representation theory
{\bf 1}, (1997), 190-206.

\bibitem[Hu]{Hu} J. E. Humphreys, {\em Representations of semisimple
    Lie algebras in the BGG category O.} AMS (2008), ISBN:
  0821846787.

\bibitem[KaV]{KaV} M. Kashiwara and M. Vergne, {\em On the
    Segal-Shale-Weil representations and harmonic polynomials}.
  Invent. Math. {\bf 44} (1978), no. 1, 1-47.

\bibitem[Kn]{Kn} A. Knapp, {\em Nilpotent orbits and some unitary
    representations of indefinite orthogonal groups.} J.
    Funct. Anal.  209 (2004), 36-100.

\bibitem[KV]{KV} A. Knapp and D. Vogan, {\em Cohomological
    induction and unitary representations}, Princeton University
  Press, New Jersey (1995).

\bibitem[Ko]{Ko} T. Kobayashi, {\em Discrete decomposability of the
    restriction of $A_{\frakq}(\lambda)$ with respect to reductive
    subgroups. III. Restriction of Harish-Chandra modules and
    associated varieties.} Invent. Math. 131 (1998), no. 2, 229-256.

\bibitem[K\O]{KO} T. Kobayashi and B. {\O}rsted, {\em Analysis of the
    minimal representation of $\rO(p,q)$ I, II, III.} Adv. Math. 180,
    No 2 (2003) 486-595. 



\bibitem[Li]{Li} J.-S. Li, {\em Singular unitary representations of
classical groups}, Invent. Math. {\bf 97} (1990), 237-255.



\bibitem[LZ]{LZ} S. T. Lee and C.-B. Zhu, {\em Degenerate principal
    series and local theta correspondence. II.} Israel J. Math. {\bf
    100} (1997), 29--59.

\bibitem[Lo]{Lo} H. Y. Loke, {\em Howe quotients of unitary
    characters and unitary lowest weight modules.} Representation
  Theory. 10 (2006), 21-47.

\bibitem[LS]{LS} H. Y. Loke and G. Savin, {\em Some unitary
    representations of orthogonal groups.}, Journal of Functional
  Analysis 255 (2008) 184-199.


\bibitem[PT]{PT} Annegret Paul and Peter Trapa. {\em Some small
    unipotent representations of indefinite orthogonal groups and the
    theta correspondence.} University of Aarhus Publication Series, 48
  (2007), 103-125.

\bibitem[Pz]{Pz} T. Przebinda, {\em The duality correspondence of
        infinitesimal characters.}  Colloq. Math. 70. (1996), no. 1,
      93--102.
      
\bibitem[Sc]{Sc} W. Schmid {\em Some Properties of Square-Integrable
Representations of Semi-simple Lie Groups}, Ann. of Math. {\bf 102} (1975),
535-564.
  
\bibitem[Ta]{Ta} U-Liang Tang, {\em The structure of Howe quotients of
    unitary lowest weight modules.} Preprint.

\bibitem[T]{T} Peter E. Trapa, {\em Some small unipotent
    representations of indefinite orthogonal groups.} Journal of
    Functional Analysis {\bf 213} (2004) 290-320.

\bibitem[Vo]{Vo} D. Vogan, {\em Representations of real reductive
      Lie groups}, Birkh\"{a}user, Boston-Basel-Stuttgart, (1981).



\bibitem[VZ]{VZ} D. A. Vogan and G. J. Zuckerman, {\em Unitary
    representations with nonzero cohomology.} Compositio Math. 53
  (1984), no. 1, 51-90.

\bibitem[W1]{Wa1} N. R. Wallach {\em Real reductive groups I.}
  Academic press, (1988).

\bibitem[W2]{Wa2} N. R. Wallach, {\em Transfer of unitary
    representations between real forms.}  Representation theory and
    analysis on homogeneous spaces (New Brunswick, NJ, 1993),
    181-216, Contemp. Math., 177, Amer. Math. Soc., Providence, RI,
    1994.

\bibitem[WZ]{WZ} N. R. Wallach and C.-B. Zhu, {\em Transfer of unitary
    representations between real forms.} Asian Jour. Math. 8, No. 4,
  (2004), 861-880.

\bibitem[Z1]{Z} C.-B. Zhu, {\em Invariant distributions of classical
groups.} Duke Math. J. {\bf 65} (1992), 85-119.

\bibitem[Z2]{Zhuscaler} C.-B. Zhu, {\em Representations with scalar
    $K$-types and applications.} Israel Jour. Math. {\bf 135} (2003), 111-124.





\end{thebibliography}
\end{document}